\documentclass[oneside]{llncs}
\makeatletter
\@twosidefalse
\@mparswitchfalse
\makeatother

\usepackage{amssymb,stmaryrd,mathrsfs,dsfont,mathtools,hyperref}

\usepackage{enumitem}

\usepackage[ruled,linesnumbered]{algorithm2e}
\usepackage{algpseudocode}
\usepackage{tikz}
\tikzstyle{vertex}=[circle, draw, inner sep=0pt, minimum size=3pt, fill=black]
\newcommand{\vertex}{\node[vertex]}

\usepackage{geometry}
\usepackage{nicematrix}
\usetikzlibrary{decorations.pathreplacing}

\begin{document}

\frontmatter         

\pagestyle{headings}  

\mainmatter             

\title{On the Conjecture of the Representation Number of Bipartite Graphs}

\titlerunning{On the Representation Number of Bipartite Graphs}

\author{Khyodeno Mozhui \and K. V. Krishna}

\authorrunning{Khyodeno Mozhui \and K. V. Krishna}

\institute{Indian Institute of Technology Guwahati, Assam, India,\\
	\email{k.mozhui@iitg.ac.in};
	\email{kvk@iitg.ac.in}}

\maketitle         

\begin{abstract}
While the problem of determining the representation number of an arbitrary word-representable graph is NP-hard, this problem is open even for bipartite graphs.  The representation numbers are known for certain bipartite graphs including all the graphs with at most nine vertices. For bipartite graphs with partite sets of sizes $m$ and $n$, Glen et al. conjectured that the representation number is at most $\lceil \frac{m+n}{4}\rceil$, where $m+n \ge 9$. 

\qquad In this paper,  we show that every bipartite graph is $\left( 1+ \lceil \frac{m}{2} \rceil  \right)$-representable, where $m$ is the size of its smallest partite set. Furthermore, if $m$ is odd then we prove that the bipartite graphs are $\lceil \frac{m}{2} \rceil $-representable. Accordingly, we establish that the conjecture by Glen et al. holds good for all bipartite graphs leaving the bipartite graphs whose partite sets are of equal and even size. In case of the bipartite graphs with partite sets of equal and even size, we prove the conjecture for certain subclasses using the neighborhood inclusion graph approach.   
\end{abstract}


\section{Introduction}

A word over finite set $A$ is a finite sequence of elements of $A$, and written by juxtaposing the symbols of the sequence. A subword $u$ of a word $w$, denoted by $u \ll w$, is a subsequence of the sequence $w$. For instance, $abaa \ll aabbaba$. When $u$ is not a subword of $w$, we write $u \not\ll w$. Given a word $w$ over $A$ and a subset $B$ of $A$, the subword obtained from $w$ by removing all letters of $A$ except the letters of $B$ is denoted by $w|_{B}$. For example, if $x = abbcbaccb$ then $x|_{{\{a, b\}}} = abbbab$.  The reversal of a word $w$, denoted by $w^R$, is a word in which the letters are written in reverse order. For the previous example, $x^R = bccabcbba$.  We say the letters $a$ and $b$ alternate in a word $w$, if $w|_{\{a, b\}}$ is either $abab\cdots$ or $baba\cdots$ (of even or odd length). A word $w$ is said to be $k$-uniform if the number of occurrences of each letter in $w$ is $k$.

A simple graph (i.e., without loops or parallel edges) $G$ is said to be word-representable if it can be represented by a word $w$ over its vertices such that any two vertices in $G$ are adjacent  if and only if they alternate in $w$. The word $w$ is called a word-representant of the graph $G$. This notion covers many important classes of graphs, including comparability graphs, cover graphs, circle graphs and 3-colorable graphs. The word-representable graphs are also interesting from the point of view of computation. For example, the maximum clique problem is solvable in polynomial time for word-representable graphs. In fact, determining whether a graph has a word-representant is NP-complete. Since the introduction \cite{kitaev08order}, the word-representable graphs have been extensively studied and the literature has several important contributions to their theory along with their connections to various concepts. For a comprehensive introduction, one may refer to the monograph by Kitaev and Lozin \cite{kitaev15mono}.

A word-representable graph $G$ has a $k$-uniform word-representant, for some $k$ \cite{kitaev08}. In which case, the graph is known as $k$-word-representable or $k$-representable. The representation number of a word-representable graph $G$, denoted by $\mathcal{R}(G)$, is the minimum number $k$ such that the graph is represented by a $k$-uniform word. It can be observed that the complete graphs have the representation number one. It was proved that the circle graphs have representation number at most two, while determining the representation number of an arbitrary word-representable graph is NP-hard  \cite{halldorsson11}. Nevertheless, the representation numbers of certain graphs were determined in the literature. For example,  the representation number of any graph belonging to prisms (\cite{kitaev13}), melon graphs (\cite{KM_KV_2024}) and split graphs (\cite{tithi_2025}) is at most three, while the representation number of an $n$-vertex crown graph is $\lceil \frac{n}{4} \rceil$ (\cite{glen18}) and it is $O(\frac{\log k}{\log \log k})$ for a $k$-cube (\cite{hefty_2024}).

It was shown in \cite{kitaev08order} that bipartite graphs are word-representable. There have been attempts to investigate the representation number of bipartite graphs, while it is determined for a few subclasses like trees, crown graphs and hypercubes. Glen et al. conjectured in \cite{glen18} that the representation number of a bipartite graph with partite sets of sizes $m$ and $n$ is at most $\lceil \frac{m+n}{4} \rceil$, for $m+n \ge 9$. Following which, in \cite{Akgun_2019}, Akg\"{u}n et al. conjectured that the crown graph\footnote{A graph which is obtained by removing a perfect matching from a complete bipartite graph $K_{n,n}$.} $H_{n,n}$ has the highest representation number among all bipartite graphs on $2n$ vertices.

In an attempt to prove the conjecture by Glen et al. on the representation number of bipartite graphs, in this paper, first we show that the representation number of a bipartite graph is at most $1 + \lceil \frac{m}{2}\rceil$, where $m$ is the size of its smallest partite set. Further, we improve this bound for various special classes of bipartite graphs. Using these results, we prove that the conjecture holds good for all bipartite graphs, leaving the bipartite graphs with partite sets of equal and even size. Consequently, we partially establish the conjecture of \cite{Akgun_2019} that a crown graph $H_{n,n}$ has the highest representation number among the bipartite graphs on $2n$ vertices, when $n$ is odd.

The representation number of a disconnected graph is equal to the maximum of the representation numbers of its connected components \cite{kitaev08}. Hence, it is sufficient to focus on determining the representation number of a connected graph. A graph is said to be  reduced if no two vertices have the same neighborhood. In this work, the neighborhood (i.e., the set of all adjacent vertices) of a vertex $a$ is denoted by $N(a)$. The representation number of a bipartite graph equals the representation number of its reduced graph that is obtained by considering only those vertices with distinct neighborhoods (and removing the remaining vertices) \cite{kitaev13}. Throughout this paper, unless specified otherwise, a bipartite graph is considered to be connected and reduced, and it is denoted by $G = (A \cup B, E)$, where $A$ and $B$ are partite sets of sizes $m$ and $n$, respectively, with $m \le n$. 

The paper has been organized as follows. In Section \ref{const_single}, we prove that every bipartite graph is $\left( 1+ \lceil \frac{m}{2} \rceil  \right)$-representable. Accordingly, we verify the conjecture by Glen et al. for the bipartite graphs with $n \ge m+3$. If $n = m+1$ or $m+2$, we also show that the conjecture follows from the same result when $m$ is even. In order to address the conjecture for the remaining cases, in Section \ref{odd_size}, we focus on the bipartite graphs with a partite set of odd size and show that the conjecture holds for these graphs. In case of bipartite graphs with partite sets of equal and even size, we prove the conjecture for certain  subclasses using the neighborhood inclusion graph approach in Section \ref{neighborhood}.      

\section{Bipartite graphs are $\left( 1+ \lceil \frac{m}{2} \rceil  \right)$-representable}\label{const_single}

In this section, we construct a uniform word  that represents a given bipartite graph. As a result, we obtain an upper bound on its representation number.

Let $G$ be a bipartite graph with bipartition $A = \{a_1,a_2,\ldots,a_m\}$ and $B= \{b_1,b_2,\ldots,b_n\}$, where $m\le n$.  Through the following steps, we construct a $\left( 1+ \lceil \frac{m}{2} \rceil  \right)$-uniform word-representant of $G$. 
 
\begin{enumerate}[label=\Roman*.]
	\item  Partition the partite set $A$ into pairs. If $m$ is odd then consider the set $A \cup \{a_{m+1}\}$, where $N(a_{m+1}) = \varnothing$. Accordingly, let $m$ be even and partition the set $A$ into $ \frac{m}{2} $ pairs, say, 
	$$(a_1,a_2), (a_3,a_4), \ldots, (a_{m-1},a_m).$$
	
	\item {\bf Word $w_{(i,j)}$:} For any pair $(a_i,a_j)$, construct a word $w_{(i,j)}$ over the set $\{a_i, a_j\} \cup B$ given by
	\[w_{(i,j)} = B_{*i}   a_i   B_i  a_j   w_{(i,j)}'  a_i   B_j'   a_j   B_{*j}'\]
	 where,
	\begin{enumerate}[label=\roman*.]		
		\item $B_{*i}B'_{*j}$: A permutation on the vertices of $B$ which are not adjacent to any vertex of $\{a_i, a_j\}$.
		\item $B_{i}$: A permutation on the neighbors of $a_i$ but not of $a_j$, i.e.,  $N(a_{i}) \setminus N(a_{j})$.
		\item $w_{(i,j)}'$: A permutation on the common neighbors of $a_i$ and $a_j$, i.e., $N(a_{i}) \cap N(a_{j})$. 
		\item $B_{j}'$: A permutation on the neighbors of $a_j$ but not of $a_i$, i.e.,  $N(a_{j}) \setminus N(a_{i})$. 
	\end{enumerate}
	We call the words $B_{*i}, B_i, w_{(i,j)}', B_j'$ and $B_{*j}'$ as blocks of the word $w_{(i,j)}$.
		
	\item  {\bf Construction of $D_{({p}:q)}$:} 
	Consider two consecutive pairs, say $(a_p, a_{p+1})$ and $(a_{q-1}, a_{q})$, for $1 \le p < m-1$; i.e., $p + 2 = q - 1$. If $p = m-1$ then consider $q = 2$, i.e., the last and first pairs. Avoiding these two pairs, from the remaining pairs, consider the word with second components in the sequence followed by the word with the first components in the sequence to construct the word $D_{(p:q)}$. That is, for various values of $p$, the words $D_{p, q}$ are as per the following:
	\[D_{(p:q)}=  
	\begin{cases}
		{a_6}{a_8}  \cdots  {a_{m}}{a_{5}}{a_{7}} \cdots  {a_{m-1}}, \text{ if } p = 1\; (\text{by removing first two pairs});\\
		{a_2}{a_4}  \cdots  {a_{m-4}}{a_{1}}{a_{3}} \cdots  {a_{m-5}}, \text{ if } p = m-3\; (\text{by removing last two pairs});\\
		{a_4}{a_6}  \cdots  {a_{m-2}}{a_{3}}{a_{5}} \cdots  {a_{m-3}}, \text{ if } p = m-1\; (\text{by removing the last and first pairs});\\
		{a_2}{a_4}  \cdots  {a_{p-1}}{a_{q+2}} \cdots  {a_{m}}   {a_{1}}{a_{3}} \cdots  {a_{p-2}}{a_{q+1}} \cdots a_{m-1}, \text{else.} 
	\end{cases}
	\]

	\item {\bf Construction of $w_0$:} Consider  the reversal of the word $w_{(1,2)}$ restricted to the vertices of $B$. Then  construct the permutation $w_0$ on the set $A \cup B$ by
	$$w_0 = a_2  a_1  w_{(1,2)}^R|_{B}  a_4   a_3   \cdots   {a_m}  a_{m-1}.$$
	
	\item {\bf Construction of $w$:} Finally, construct a $\left( 1 +  \frac{m}{2}  \right)$-uniform word
	\[w = w_{(1,2)}  D_{(1:4)}    w_{(3,4)}  D_{(3:6)}w_{(5,6)}   \cdots   w_{(m-1,m)}  D_{(m-1:2)}  w_0\]	
	over the set $A \cup B$.	
\end{enumerate}

\begin{remark}
	Note that if $m$ is odd, we construct the word $w$ over the vertex set $A\cup B \cup \{a_{m+1}\}$. It is evident that the word obtained by removing all the occurrences of the vertex $a_{m+1}$ from $w$ is a word-representant of the bipartite graph $G$.
\end{remark}

A demonstration of the construction is given in Appendix \ref{demo}.

\subsubsection{Observations and Remarks.}
We present some remarks on the occurrences of each vertex of $G$ in the word $w$ for better understanding on the construction of $w$. Let $a^{(i)}$ denote the $i$th occurrence of a vertex $a$ in the word $w$ (from left to right).

\begin{remark}\label{place_b}
	Each vertex $b \in B$, we have, for all $ 1 \le s \le \frac{m}{2}$,  $b^{(s)} \ll w_{(2s-1,2s)}$, and $b^{(1+ \frac{m}{2})} \ll w_0$. 
\end{remark}

\begin{remark}\label{place_a}
	For all $1 \le s \le \frac{m}{2}$, the word $w_{(2s-1,2s)}$ contains only the vertices $a_{2s-1}$ and $ a_{2s}$ from the partite set $A$. For each $a \in A$,
	\begin{itemize}
		\item[--] $a \in \{a_{1}, a_{2} \}$. We have
		\begin{align*}
			a^{(1)} a^{(2)} &\ll w_{(1,2)}\\
			\text{for all } 3 \le i \le \frac{m}{2}, \quad  a^{(i)} &\ll \begin{cases}
				D_{(2i-3:2i)}, & \text{ if } i\ne \frac{m}{2};\\
				D_{(m-1:2)}, &\text{ if } i = \frac{m}{2}
			\end{cases}\\
			a^{(1+ \frac{m}{2})} &\ll w_0.
		\end{align*}
	
		\item[--] $a \in \{a_{2s-1}, a_{2s} \}$, where $2 \le s \le \frac{m}{2}$. We have
		\begin{align*}
			\text{for all } i \in \{1,2, \ldots, s-2,s+1, \ldots, \frac{m}{2}\}, \quad  a^{(i)} &\ll \begin{cases}
				D_{(2i-1:2i+2)}, & \text{ if } i \ne \frac{m}{2};\\
				D_{(m-1:2)}, & \text{ if } i = \frac{m}{2}
			\end{cases}\\
			a^{(s-1)} a^{(s)} &\ll w_{(2s-1,2s)}\\
			a^{(1+ \frac{m}{2})} &\ll w_0.
		\end{align*}		
	\end{itemize} 
	
\end{remark} 

The remarks \ref{place_b} and \ref{place_a} are very handy in proving that the word $w$ represents $G$ in the following theorem.

\begin{theorem}\label{main-1}
	The word $w$ represents the bipartite graph $G$.
\end{theorem}

\begin{proof}
	We show that for any two vertices $a, b \in A \cup B$, $a$ and $b$ are adjacent if and only if $a$ and $b$ alternate in $w$, so that $w$ is a word-representant of  the bipartite graph $G$.
		
	Suppose $a$ and $b$ are adjacent in $G$. Let $a \in A$ and $b \in B$. We deal this in the following two cases:
	\begin{itemize}
		\item[-]Case 1: $a \in \{a_1,a_2\}$: From remarks \ref{place_b} and \ref{place_a}, we have
		\begin{align*}
			a^{(1)}b^{(1)}a^{(2)} &\ll w_{(1,2)}  \\
			\text{for all } \ 3 \le i \le \frac{m}{2},\quad  b^{(i-1)}  &\ll w_{(2i-3,2i-2)} \text{ and }  a^{(i)} \ll \begin{cases}
				D_{(2i-3:2i)}, & \text{ if } i\ne \frac{m}{2};\\
				D_{(m-1:2)}, &\text{ if } i = \frac{m}{2}
			\end{cases} \\
			b^{(\frac{m}{2})} &\ll w_{{(m-1,m)}} \\
			a^{(1+\frac{m}{2})} b^{(1+\frac{m}{2})} &\ll w_0
		\end{align*}
		so that $a^{(1)} b^{(1)} a^{(2)} b^{(2)}  \cdots  a^{(1+\frac{m}{2})} b^{(1+\frac{m}{2})} \ll w$.
		
		\item[-]Case 2: $a \in \{a_{2s-1}, a_{2s}\}$, for $2 \le s \le \frac{m}{2}$. From remarks \ref{place_b} and \ref{place_a}, we have 
		\begin{align*}	
		\text{for each } \ i \in \{1,\ldots,s-2,s+1, \ldots, \frac{m}{2}\}, \quad b^{(i)} &\ll w_{(2i-1,2i)} \text{ and } a^{(i)}  \ll \begin{cases}
			D_{(2i-1:2i+2)}, & \text{ if } i \ne \frac{m}{2};\\
			D_{(m-1:2)}, & \text{ if } i = \frac{m}{2}
		\end{cases}\\
		b^{(s-1)} &\ll w_{(2s-3,2s-2)} \\
		a^{(s-1)} b^{(s)} a^{(s)} &\ll w_{(2s-1,2s)} \\
		b^{(1+\frac{m}{2})} a^{(1+\frac{m}{2})} & \ll w_0
		\end{align*}
		so that $\ b^{(1)}  a^{(1)} b^{(2)}  a^{(2)}  \cdots  b^{(1+\frac{m}{2})}  a^{(1+\frac{m}{2})} \ll w$.
	\end{itemize}
	Hence, $a$ and $b$ alternate in the word $w$.
	
	Conversely, suppose $a$ and $b$ are not adjacent in $G$. We show that, for some $j \ge 1$, one of the following words is a subword of $w$: $a^{(j)}b^{(j)} b^{(j+1)}a^{(j+1)}$, $b^{(j)}a^{(j)} a^{(j+1)} b^{(j+1)} $, $a^{(j)} a^{(j+1)} b^{(j)} b^{(j+1)}$,   $b^{(j)}b^{(j+1)} a^{(j)} a^{(j+1)}$. We handle this in the following three cases:
	\begin{itemize}
		\item[-] Case 1: $a,b \in A$. We have the following subcases:
		\begin{itemize}
			\item Subcase 1.1: $a, b$ belong to the same pair. For $1 \le s \le \frac{m}{2}$, let $a = a_{2s-1}$ and $b = a_{2s}$. 
			\begin{enumerate}[label=\roman*)]
				\item If $s = 1$ then note that $a^{(2)}b^{(2)} = a_1^{(2)}a_2^{(2)} \ll w_{(1,2)}$. Further, 
				\[a_2^{(3)}a_1^{(3)} \ll
				\begin{cases}
					w_0, & \text{ for } m \le 4;\\
						D_{(3:6)}, & \text{ for } m \ge  6,
				\end{cases}\]
				so that $a^{(2)}b^{(2)}b^{(3)}a^{(3)} \ll w$. 
				
				\item If $s = 2$ then note that $a^{(2)}b^{(2)} = a_3^{(2)}a_4^{(2)} \ll w_{(3,4)}$. Further, 
				\[a_4^{(3)}a_3^{(3)} \ll
				\begin{cases}
					w_0, & \text{ for } m \le 4;\\
					D_{(5:2)}, & \text{ for } m = 6;\\
					D_{(5:8)}, & \text{ for } m > 6,
				\end{cases}\]
				so that $a^{(2)}b^{(2)}b^{(3)}a^{(3)} \ll w$.
				\item If $s > 2$ then we have, $b^{(s-2)} a^{(s-2)}a^{(s-1)} b^{(s-1)} \ll D_{(2s-5:2s-2)} w_{(2s-1,2s)}\ll w$.
			\end{enumerate}
			
			\item Subcase 1.2: $a$ and $b$ belong to different pairs. 
			\begin{enumerate}[label=\roman*)]
				\item If one of $a, b$ is in the first pair. Without loss of generality, let $a \in \{a_{1}, a_{2}\}$. Then, as $b \notin \{a_{1}, a_{2}\}$, $b$ does not appear in $w_{(1,2)}$. Accordingly, we have $$a^{(1)}a^{(2)} \ll w_{(1,2)} \text{ so that }\; a^{(1)}a^{(2)}b^{(1)}b^{(2)} \ll w.$$
				\item If $a \in \{a_{2s-1}, a_{2s}\}$, for $2 \le s \le \frac{m}{2}$ then consider the following cases depending on whether $b$ is in the subsequent/preceding  pair or not:
				\begin{enumerate}
					\item If $b \in \{a_{2s+1}, a_{2s+2}\}$ then 
					\begin{align*}
						b^{(s-1)} &\ll D_{(2s-3:2s)}\\
						a^{(s-1)} a^{(s)} &\ll w_{(2s-1, 2s)}\\
						b^{(s)} b^{(s+1)} &\ll w_{(2s+1, 2s+2)}
					\end{align*}
					so that $b^{(s-1)}a^{(s-1)} a^{(s)}b^{(s)} \ll w$.
					
					\item If $b \notin \{a_{2s-3}, a_{2s-2},a_{2s+1}, a_{2s+2}\}$ then  we have $$b^{(s-1)} a^{(s-1)} a^{(s)} b^{(s)}\ll \begin{cases}
						 D_{(2s-3:2s)} w_{(2s-1,2s)} D_{(2s-1:2)}, &\text{ for } s=\frac{m}{2};\\
						  D_{(2s-3:2s)} w_{(2s-1,2s)} D_{(2s-1:2s+2)}, &\text{ for } s < \frac{m}{2},
					\end{cases}$$
					so that $b^{(s-1)} a^{(s-1)} a^{(s)} b^{(s)}\ll w$.
				\end{enumerate}
			\end{enumerate}
		\end{itemize} 
		
		\item[-] Case 2: $a \in A$ and $b \in B$. Let $a \in \{a_{2s-1}, a_{2s}\}$, for $1 \le s\le \frac{m}{2}$. 
		
		\begin{enumerate}[label=\roman*)]
			\item If $b \in N(a_{2s-1})$ then $b^{(s)} \ll B_{2s-1}  \ll w_{(2s-1,2s)} $.
			\begin{enumerate}[label=(\alph*)]
				\item  If $s = 1$ then $b^{(1)} a^{(1)} a^{(2)}   \ll w_{(1,2)}$ and $b^{(2)} \ll w_{(3,4)}$. Thus, $b^{(1)} a^{(1)} a^{(2)}  b^{(2)} \ll w$.
				\item If $s > 1$ then $b^{(s-1)} \ll w_{(2s-3,2s-2)}$ and $b^{(s)} a^{(s-1)} a^{(s)}   \ll w_{(2s-1,2s)}$. Thus, $b^{(s-1)} b^{(s)} a^{(s-1)} a^{(s)}   \ll w$. 
			\end{enumerate}
			
			\item If $b \in N(a_{2s})$ then $b^{(s)} \ll B_{2s}' \ll w_{(2s-1,2s)}$.
			\begin{enumerate}[label=(\alph*)]
				\item  If $s = 1$ then 	$a_1^{(1)} a_1^{(2)} b^{(1)} \ll w_{(1,2)} \text{ and } b^{(2)} \ll w_{(3,4)}$, so that $a^{(1)} a^{(2)} b^{(1)}b^{(2)} \ll w$.
				\item If $s > 1$ then $b^{(s-1)} \ll w_{(2s-3,2s-2)} \text{ and } a^{(s-1)} a^{(s)} b^{(s)}  \ll w_{(2s-1,2s)}$, so that $$b^{(s-1)}a^{(s-1)} a^{(s)} b^{(s)}  \ll w.$$
			\end{enumerate}	
			
			\item If $b \not\in N(a_{2s-1}) \cup N(a_{2s})$ then either $b^{(s)} \ll B_{*(2s-1)}$ or $b^{(s)} \ll B_{*(2s)}'$. 
			\begin{enumerate}[label=(\alph*)]
				\item If $s = 1$ then $b^{(1)} a^{(1)} a^{(2)}  \ll w_{(1,2)}$ or $ a^{(1)} a^{(2)} b^{(1)} \ll w_{(1,2)}$, and $b^{(2)} \ll w_{(3,4)}$. Thus, $$b^{(1)} a^{(1)} a^{(2)} b^{(2)} \ll w \text{ or } a^{(1)} a^{(2)} b^{(1)} b^{(2)}  \ll w.$$ 
				\item  If $s > 1$ then  $b^{(s-1)} \ll w_{(2s-3,2s-2)}$, and $b^{(s)}a^{(s-1)} a^{(s)}  \ll w_{(2s-1,2s)}$ or $a^{(s-1)} a^{(s)} b^{(s)} \ll w_{(2s-1,2s)}$.	 Thus, $$b^{(s-1)} b^{(s)} a^{(s-1)} a^{(s)} \ll w \text{ or } b^{(s-1)}a^{(s-1)} a^{(s)} b^{(s)} \ll w.$$
			\end{enumerate}
		\end{enumerate}
		
		\item[-] Case 3: $a,b \in B$.	Note that $w= w_{(1,2)}uw_0$, for some $u$. Clearly, $w_{(1,2)}|_{B} \ll w_{(1,2)}$, and by construction of $w_0$ we have, ${w_{(1,2)}^R}|_B \ll w_0$. Since $a$ and $b$ occur exactly once in $w_{(1,2)}$ (cf. Remark \ref{place_b}), without loss of generality, assume $ab\ll w_{(1,2)}$. Then $w|_{{\{a, b\}}} = ab \cdots ba$. Since $w$ is a uniform word, both $a$ and $b$ must occur the same number of times. Thus, two $b$'s occur consecutively in $w|_{{\{a,b\}}}$, so that $a$ and $b$ do not alternate in $w$. 
	\end{itemize}
	Hence, $a$ and $b$ do not alternate in the word $w$. 
	\qed
\end{proof}

\begin{corollary}\label{main_result}
	The representation number of a bipartite graph is at most  $1+ \lceil\frac{m}{2}\rceil$, where $m$ is the size of the smallest partite set.
\end{corollary}

\begin{theorem}\label{conj_veri_1}
	Let $G$ be a bipartite graph. Then we have the following: \begin{enumerate}
		\item If $n \ge m+3$ then $\mathcal{R}(G) \le \lceil \frac{m+n}{4} \rceil$.
		\item If $m$ is even and $n \in \{m+1,m+2\}$ then $\mathcal{R}(G) \le \lceil \frac{m+n}{4} \rceil$.
	\end{enumerate} 
\end{theorem}
\begin{proof}$\;$
	\begin{enumerate}
		\item By Theorem \ref{main_result}, if $m = 2k$ or $m = 2k-1$, for some $k \ge 1$, we have $$ \mathcal{R}(G) \le 1 + \lceil \frac{m}{2} \rceil = 1 + k.$$
		
		On the other hand, if $m = 2k$  then
		$\lceil \frac{m+n}{4} \rceil \ge  \lceil \frac{4k+ 3}{4} \rceil = k + 1.$ Similarly, if $m = 2k -1$, $\lceil \frac{m+n}{4} \rceil \ge  \lceil \frac{4k+ 1}{4} \rceil = k + 1.$ 
		Hence, in any case, $\mathcal{R}(G) \le 1 + k \le \lceil \frac{m + n }{4} \rceil$.
		
		\item Let $m = 2k$, for some $k \ge 1$. Note that $\mathcal{R}(G) \le 1 + k.$ If $n = m + 1$ then $\lceil \frac{m+n}{4} \rceil = \lceil \frac{4k+ 1}{4} \rceil = k + 1.$ Similarly, if $n = m + 2$ then
		$\lceil \frac{m+n}{4} \rceil =  \lceil \frac{4k+ 2}{4} \rceil = k + 1$ so that $\mathcal{R}(G) \le 1 + k = \lceil \frac{m + n }{4} \rceil$.
	\end{enumerate}\qed
\end{proof}

\section{Bipartite graphs with partite set of odd size} \label{odd_size}

From results given in \cite{Akgun_2019,glen18}, we observe that, for graphs with at most nine vertices, the representation number of a bipartite graph is at most $1 + \lceil \frac{m}{2} \rceil$ and the upper bound is tight. In this section, for $m \ge 5$, we show that a bipartite graph is $\lceil \frac{m}{2} \rceil$-representable if $m$ is odd.

\begin{theorem}\label{odd_result}
	Let $G$ be a bipartite graph with bipartition $A = \{a_1,a_2, \ldots, a_m\}$ and $B = \{b_1, b_2, \ldots, b_n\}$, where $ 5 \le  m \le n$. If $m$ is odd then $G$ is $\lceil \frac{m}{2} \rceil$-representable.
\end{theorem}
\begin{proof}

	Let $P_1 = \{b \in B| \ \deg(b) = m-1\}$ and $P_2 = B \setminus P_1$. Without loss of generality, assume that $P_1 = \{b_1, b_2, \ldots, b_r\}$ and $P_2 =\{b_{r+1}, b_{r+2}, \ldots, b_n\}$. Since $G$ is reduced (i.e., each vertex of $G$ has distinct neighborhood), for each $1 \le i \le r$, there exists $a_i \in A$ such that $b_i$ is non-adjacent to $a_i$. Clearly, the remaining vertices of $A$, say $a_{r+1}, \ldots, a_m$, are adjacent to each vertex of $P_1$. 
	
	Suppose $P_2 = \varnothing$. Then  the graph $G$ is the crown graph $H_{m,m}$ and from \cite[Theorem 7]{glen18}, $G$ is $\lceil\frac{m}{2}\rceil$-representable.
	
	Suppose $P_2 \ne \varnothing$. Since $m$ is odd, consider the set $A \cup \{a_{m+1}\}$ with a new isolated vertex $a_{m+1}$  and partition it into pairs, i.e., $ M_1 = \{  (a_1,a_2), \ldots, (a_{m-2},a_{m-1}), (a_m,a_{m+1})\} $. A vertex $c$ is said to be adjacent to a pair $(a_i,a_{j})$ if $c$ is adjacent to $a_i$ or $a_{j}$.  
	
	Recall that in the construction of the word $w_{(i,i+1)}$ given in Section \ref{const_single}, the arrangement of the vertices of $B$ within each block among $B_{*i}, B_{*(i+1)}', B_{i}, B_{(i+1)}'$, and $w_{(i,i+1)}'$  are in any order. For each pair $(a_i,a_{i+1})$, we now construct the word $w_{(i,i+1)}$ following the process given in Section \ref{const_single}, but by arranging the vertices of $B$ in each of the blocks in a specific order. 
	
	Note that $M_1$ has at least three pairs, as $m \ge 5$.  For each $b \in B$, case wise, first we implement the following steps:
	\begin{enumerate}
		\item If $b$ is non-adjacent to (at least) two pairs in $M_1$, say $(a_i, a_{i+1})$ and $(a_j, a_{j+1})$ then put $b$ in $B_{*i}$ of $w_{(i,i+1)}$ and put $b$ in $B_{*(i+1)}'$ of $w_{(j,j+1)}$.   
		\item If $b$ is non-adjacent to only one pair $(a_i, a_{i+1})$ in $M_1$, where $i \ne m$ then put $b$ in $B_{*i}$ of $w_{(i,i+1)}$ if $b$ is in $B_{t+1}'$ of some other word $w_{(t,t+1)}$. Otherwise, put $b$ in $B_{*(i+1)}'$ of $w_{(i,i+1)}$.
		\item If $b$ is non-adjacent to only $(a_m, a_{m+1})$ then put $b$ in $B_{*(m+1)}'$ of $w_{(m,m+1)}$ if $b$ is in $B_{t}$ of some other word $w_{(t,t+1)}$. Otherwise, put $b$ in $B_{*m}$ of $w_{(m,m+1)}$.
	\end{enumerate}

	Then we focus on arranging the vertices of $B$ within each block of $w_{(i,i+1)}$, for all $i$. Observe that if two vertices $b$ and $b'$ of $B$ occur in the same block of $w_{(i,i+1)}$ then $b$ and $b'$ must occur in different blocks of some word $w_{(t,t+1)}$, as $G \cup \{a_{m+1}\}$ is reduced. 
	\begin{itemize}
		\item For every pair of vertices $b$ and $b'$ within the same block of $w_{(1,2)}$,  if $b'$ occur before $b$ in different blocks of some word $w_{(t,t+1)}$ then put $b$ before $b'$ in the block of $w_{(1,2)}$. Else, put $b'$ before $b$ in the block of $w_{(1,2)}$.
		\item For $i > 1$, for every pair of vertices $b$ and $b'$ within the same block of $w_{(i,i+1)}$,  if $b'$ occur before $b$ in the word $w_{(1,2)}$ then put $b$ before $b'$ in the block of $w_{(i,i+1)}$. Else, put $b'$ before $b$ in the block of $w_{(i,i+1)}$. 	
	\end{itemize}
	By Theorem \ref{main_result}, the following $(1 + \lceil\frac{m}{2}\rceil)$-uniform word $w$ represents $G \cup \{a_{m+1}\}$.
	\[w  = w_{(1,2)} D_{(1:4)} w_{(3,4)} \cdots w_{(m,m+1)} D_{(m:2)}w_0\]

	Recall that $w_0$ is a permutation on the set $A \cup B \cup \{a_{m+1}\}$. We show that the following $\lceil \frac{m}{2} \rceil$-uniform word $w_1$ obtained by dropping the suffix $w_0$ from $w$ represents the graph $G \cup \{a_{m+1}\}$. 
	\[w_1 = w_{(1,2)} D_{(1:4)} w_{(3,4)} \cdots w_{(m,m+1)} D_{(m:2)}\]
	
	Note that if any two vertices are adjacent in $G \cup \{a_{m+1}\}$ then they   alternate in $w$ and hence they also alternate in $w_1$. Suppose two vertices of $G \cup \{a_{m+1}\}$ are not adjacent. If both are in $A \cup \{a_{m+1}\}$ or one is in $A \cup \{a_{m+1}\}$ and the other is in $B$, observe that they do not alternate in $w$ so that they will not alternate in $w_1$, as $w_0$ does not have any role in the proof of non-adjacency for $m > 4$ in Theorem \ref{main_result}.   
	
	We only need to show that if $a, b \in B$ then $a$ and $b$ do not alternate in $w_1$. Since every vertex of $B$ occur exactly once in each word $w_{(i,i+1)}$ of $w$, we conclude the result by proving the following claim. 
	
	\textit{Claim. There exist two words $w_{(r,r+1)}$ and $w_{(s,s+1)}$ such that $ab \ll w_{(r,r+1)}$ and $ba \ll w_{(s,s+1)}$}.\\ 
	
	Let $a$ or $b$ be non-adjacent to at least two pairs in $M_1$. Suppose $a$ and $b$ belong to the same block of some word $w_{(j,j+1)}$. If $j = 1$ then $ab \ll w_{(1,2)}$ if only if    $ba \ll w_{(t,t+1)}$, for some $t \ne 1$. If $j \ne 1$ then  $ab \ll w_{(j,j+1)}$ if and only if $ba \ll w_{(1,2)}$.    Suppose $a$ and $b$  belong to different blocks in every $w_{(2s-1,2s)}$, for $1 \le s \le \frac{m}{2}$. Let $a$ be non-adjacent to $(a_r,a_{r+1})$ and $(a_s,a_{s+1})$. Without loss of generality, assume that $a \ll B_{*r} \ll w_{(r,r+1)}$ then $a \ll B_{*(s+1)}' \ll w_{(s,s+1)}$.  We have $ab \ll w_{(r,r+1)}$ and $ba\ll w_{(s,s+1)}$.	
	
	On the other hand, both $a, b$ are non-adjacent to at most one pair in $M_1$. That is, $a$ is adjacent to all pairs or $a$ is non-adjacent to exactly one pair. Also, $b$ is adjacent to all pairs or $b$ is non-adjacent to exactly one pair.  Accordingly, by symmetry between $a$ and $b$, we have the following cases: (i) both $a, b$ are adjacent to all pairs, (ii) $a$ is not  adjacent one pair, say $(a_i, a_{i+1})$, and $b$ is adjacent to all pairs (iii) $a$ is not adjacent one pair, say $(a_i, a_{i+1})$, and 
	$b$ is not adjacent one pair, say $(a_j, a_{j+1})$.
	
	Since $N(a_{m+1}) = \varnothing$, any vertex adjacent to the pair $(a_m, a_{m+1})$ must be adjacent to only $a_m$. Depending on whether $i$ or $j$ equals $m$, aforementioned three cases will yield the following. In case (i), we have $a, b \in N(a_m)$. In case (ii), clearly $b \in N(a_m)$; and if $i \ne m$ then $a \in N(a_m)$, otherwise $a \notin N(a_m)$. In case (iii), we get all four possibilities based on $a \in N(a_m)$ or not, and $b \in N(a_m)$ or not. Accordingly, by symmetry between $a$ and $b$, we prove our claim in the following three cases:

	\begin{itemize}
		\item[-] Case 1: $a, b \in N(a_m)$. In this case, $a$ and $b$ belong to the same block $B_{m}$ of $w_{(m,m+1)}$. Thus, if $ab \ll w_{(1,2)} $ then $ba \ll w_{(m,m+1)}$. Otherwise, $ab \ll w_{(m,m+1)}$ if $ba \ll w_{(1,2)}$.
			
		\item[-] Case 2: $a, b \not\in N(a_m)$. If $a$ and $b$ belong to the same block, as seen before, we have, $ab \ll w_{(m,m+1)}$ if and only if $ba \ll w_{(1,2)}$.
			Otherwise, without loss of generality, suppose $a \ll B_{*m}$ and $b \ll B_{*(m+1)}'$ of $w_{(m,m+1)}$ so that $ab \ll w_{(m,m+1)}$. Since $b \ll B_{*(m+1)}'$, as per our construction of placing the vertices of $B$ (cf. step 3), the vertex $b$ must occur in $B_t$  of some other word $w_{(t,t+1)}$. If $a$ and $b$ belong to the same block of $w_{(t,t+1)}$ then $ab \ll B_t \ll w_{(t,t+1)}$ if and only if $ba \ll w_{(1,2)}$, when $t \ne 1$. However, if $t =1$, as $M_1$ has at least three pairs, there exists $j \notin \{1,m\}$ such that $ba \ll w_{(j,j+1)}$. Else, if $a$ is not in the block $B_t$ of $w_{(t,t+1)}$ then $ba \ll w_{(t,t+1)}$.
		Thus, in any case, there exist $r,s$ such that $ab \ll w_{(r,r+1)}$ and $ba \ll w_{(s,s+1)}$.
			
		\item[-] Case 3: $a \in N(a_m)$ and $b \not \in N(a_m)$. Suppose $a$ and $b$ belong to the same block of some word $w_{(j,j+1)}$. If $j = 1$ then clearly  $ab \ll w_{(1,2)}$ if and only if $ba \ll w_{(t,t+1)}$, for some $t \ne 1$. If $j \ne 1$ then $ab \ll w_{(j,j+1)}$ if and only if $ba \ll w_{(1,2)}$. Suppose $a$ and $b$ belong to different blocks in every $w_{(2s-1,2s)}$, for $ 1 \le s \le \frac{m}{2}$. We handle this in the following two subcases.
		
		\begin{itemize}
			\item Subcase 3.1: $b \ll B_{*(m+1)}'$. In this case, $ab \ll B_{m} B_{*(m+1)}' \ll w_{(m,m+1)}$ and $b$ belongs to $B_{t}$  of some word $w_{(t,t+1)}$. In the word $w_{(t,t+1)}$, if $a$ does not belong to $B_{*t}$ then $ba \ll w_{(t,t+1)}$. Otherwise, $ab \ll w_{(t,t+1)}$. But, since $M_1$ has at least three pairs, there exists another word $w_{(s,s+1)}$ with $s \not \in \{m, t\}$ and $a \ll B_{s+1}' \ll w_{(s,s+1)}$, so that $ba \ll w_{(s,s+1)}$.
			
			\item Subcase 3.2: $b \ll B_{*m}$. Then  $ba \ll w_{(m,m+1)}$. Note that, for odd $j$ and $1 \le j < m$,  $b$ belongs to either $w_{(j,j+1)}'$ or $B_{(j+1)}'$ of $w_{(j,j+1)}$. Since all the vertices of $B$ with degree $m$ or $m-1$ are adjacent to $a_m$, the vertex $b$ has degree at most $m-2$. This implies that $b$ belongs to $B_{t+1}'$ of some word $w_{(t,t+1)}$. In the word $w_{(t,t+1)}$, if $a$ does not belong to $B_{*(t+1)}'$  of $w_{(t,t+1)}$ then $ab \ll w_{(t,t+1)}$. Otherwise, $ba \ll B_{t+1}' B_{*(t+1)}' \ll w_{(t,t+1)}$. As $a \ll B_{*(t+1)}'$ and $M_1$ has at least three pairs,  $a$ belongs to either $B_{s}$ or $w_{s(s+1)}'$ of some word $w_{(s,s+1)}$, where $s \not \in \{t,m\}$ so that $ab\ll w_{(s,s+1)}$.
		\end{itemize}
		
		Thus, $a$ and $b$ do not alternate in $w_1$.
		
	\end{itemize} 
	Hence, $w_1|_{A\cup B}$ represents the graph $G$ so that $G$ is $\lceil \frac{m}{2}\rceil$-representable. \qed
\end{proof}

\begin{corollary}\label{conj_veri_2} 
	Let $G$ be a bipartite graph with partite sets of sizes $m$ and $n$, where $5 \le m\le n$. If $m$ is odd then $\mathcal{R}(G) \le \lceil \frac{m+n}{4} \rceil$.
\end{corollary}

\begin{proof}
	In view of Theorem \ref{conj_veri_1}(1), it is sufficient to verify the statement for $n \in \{m,m+1,m+2\}$. 
	Let $m =2k+1$, for some $k \ge 2$ and $n = 2k + t$, for $1 \le t \le 3$. Then  by Theorem \ref{odd_result},
	\[\lceil \frac{m+n}{4} \rceil = \lceil \frac{4k + t + 1}{4} \rceil = k+1 = \lceil \frac{m}{2} \rceil \ge \mathcal{R}(G).\] \qed
\end{proof}

Through Theorem \ref{conj_veri_1} and Corollary \ref{conj_veri_2}, apart from the bipartite graphs with partite sets of equal and even size, we proved that the conjecture by Glen et al. on the representation number of bipartite graphs is valid for all other bipartite graphs.

\section{Neighborhood inclusion graph approach} \label{neighborhood}

In this section, we focus on bipartite graphs with partite sets of equal and even size. For certain subclasses of these graphs,  we also verify the conjecture by Glen et al.    

A set $P$ with a partial order, say $\prec$ on $P$, is called a partially ordered set (simply, a poset). The elements $a, b$ in a poset are comparable if $a \prec b$ or $b \prec a$; otherwise, they are incomparable. A chain in a poset $P$ is a subset of $P$ consisting of pairwise comparable elements in the poset. A chain cover of a poset $P$ is a collection of disjoint chains whose union is $P$. The height of $P$, denoted by $h(P)$, is the size of a largest chain in $P$. Note that a comparability graph\footnote{A graph which admits a transitive orientation is called a comparability graph.} $G$ induces a poset, denoted by $P_G$, based on the chosen transitive orientation. 

Suppose $G= (V,E)$  is a graph, with $A \subseteq V$. A neighborhood inclusion graph of $G$ with respect to $A$ is the directed graph $\mathcal{N}_A(G)=(A,E')$, where $\overrightarrow{ab} \in E'$ if and only if $N(a) \subseteq N(b)$ for all $a,b \in A$. Note that the inclusion relation is transitive, and the  graph of $\mathcal{N}_A(G)$ induces a poset. The poset $P_{\mathcal{N}_A(G)}$ induced by $\mathcal{N}_A(G)$ is denoted simply by $P_A$ in the rest of the paper.

\begin{theorem}\label{length_three}
	Let $G$ be a bipartite graph with a partite set 
	$A=\{a_1,a_2,\ldots,a_m\}$, where $m = 2k$ ($k \ge 3$), and  $h(P_A)\ge 3$. Then  $\mathcal{R}(G) \le k$.
\end{theorem}

\begin{proof}
	As $P_A$ is a poset of height at least three, there exists a chain of size three, say $\{a_1,a_2,a_3\}$ such that $a_1 \prec a_2 \prec a_3$. We now construct a $k$-uniform word-representant of $G$. We follow the same method given in Section \ref{const_single} but with the following modification. First partition the set $A$ into $k-1$ pairs given by:
	\[(a_1 \prec a_2\prec a_3, a_4), (a_5, a_6), \ldots, (a_{m-1}, a_m)\]
	Except for the first pair, as given in Section \ref{const_single}, we construct the word $w_{(i,j)}$ for each pair $(a_i, a_j)$, for $5 \le i \le m-1$ and $i$ odd. 
	Note that the first component of the first pair is the chain $a_1 \prec a_2 \prec a_3$. Considering the index of its first element, the word $w_{(i,j)}$ corresponding the first pair is denoted by $w_{(1, 4)}$ and constructed over the set $\{a_1,a_2,a_3,a_4\} \cup B$ as per the following:
	
	\[w_{(1,4)} = B_{*1}  a_3    B_3   a_2   B_2    a_1    B_1    a_4   C_{1}   a_1   C_2   a_2   C_3    a_3    B_4'   a_4   B_{*4}' \]
		where,
		\begin{enumerate}[label=\roman*.]
		\item $B_{*1}B'_{*4}$: A permutation on the vertices of $B$ which are not adjacent to any vertex of $\{a_1,a_2,a_3, a_4\}$.
		\item $B_{1}$: A permutation on the neighbors of $a_1$ but not of  $a_4$, i.e.,  $N(a_{1}) \setminus  N(a_4)$.
		\item $B_{2}$: A permutation on the neighbors of $a_2$ but not of  $a_1$ and $a_4$, i.e.,  $N(a_{2}) \setminus ( N(a_1) \cup N(a_4))$.
		\item $B_{3}$: A permutation on the neighbors of $a_3$ but not of  $a_2$ and $a_4$, i.e.,  $N(a_{3}) \setminus ( N(a_2) \cup N(a_4))$.
		\item $C_1$: A permutation on the common neighbors of  $a_1$ and $a_4$, i.e., $N(a_{1}) \cap N(a_{4})$. 
		\item $C_2$: A permutation on the common neighbors of  $a_4$ and $a_2$ but not of $a_{1}$, i.e., $(N(a_{2}) \cap N(a_{4})) \setminus N(a_{1})$. 
		\item $C_3$: A permutation on the common neighbors of  $a_4$ and $a_3$ but not of $a_{2}$, i.e., $(N(a_{3}) \cap N(a_{4})) \setminus N(a_{2})$. 
		\item $B_{4}'$: A permutation on the neighbors of $a_4$ but not of $a_3$, i.e.,  $N(a_{4}) \setminus N(a_{3})$. 
	\end{enumerate}
	Similar to $w_0$ constructed in Section \ref{const_single}, we construct the permutation $u_0$ on the set $A \cup B$ by
	$$u_0 = a_4  a_1  a_2   a_3   w_{(1,4)}^R|_{B}  a_6  a_5  \cdots   {a_m}  a_{m-1}.$$ 
	Finally, we construct a $k$-uniform word
	\[w_2 = w_{(1,4)}  D_{(1:6)}   w_{(5,6)}   D_{(5:8)}   w_{(7,8)}   \cdots   w_{(m-1,m)}   D_{(m-1:4)}  u_0\]	
	over the set $A \cup B$, where the word  $D_{(p,q)}$ is constructed as per the definition given in Section \ref{const_single} for consecutive pairs, in which whenever the index $1$ is referred the word $a_1a_2a_3$ will be placed together. For example, $D_{(5,8)} = a_4a_{10}a_{12}\cdots a_ma_1a_2a_3a_9a_{11} \cdots a_{m-1}$.

	We show that for any two vertices $a, b \in A \cup B$, $a$ and $b$ are adjacent in $G$ if and only if $a$ and $b$ alternate in $w_2$, so that $w_2$ is a word-representant of  the bipartite graph $G$. The proof is similar to the proof of Theorem \ref{main-1}, by considering the construction given here.
	
	Suppose $a$ and $b$ are adjacent in $G$. Let $a \in A$ and $b \in B$. We deal this in the following three cases:
	\begin{itemize}
		\item[-] Case 1: $a \in \{a_1,a_2, a_3,a_4\}$: We have
		\begin{align*}
			a^{(1)}b^{(1)}a^{(2)} &\ll w_{(1,4)}\\
		\text{for all } \ 3 \le i \le k-1, \quad	b^{(i-1)} &\ll w_{(2i-1,2i)} \text{ and } a^{(i)} \ll D_{(2i-1:2i+2)}  \\
			b^{(k -1)} &\ll w_{{(m-1,m)}}\\
			a^{( k)} b^{( k)} &\ll u_0
		\end{align*}
		so that $a^{(1)}  b^{(1)}  a^{(2)}  b^{(2)}   \cdots  a^{(k)}   b^{(k)}$ is a subword of $w_2$.
		\item[-] Case 2: $a \in \{a_{5}, a_{6}\}$: We have
		\begin{align*}
			b^{(1)} &\ll w_{(1,4)}\\
			 a^{(1)} b^{(2)} a^{(2)} &\ll w_{(5,6)}\\
			\text{for all} \ 4 \le i \le k-1, \quad b^{(i-1)}	&\ll w_{(2i-1,2i)} \text{ and } 	a^{(i-1)} \ll D_{(2i-1:2i+2)} \\
			b^{(k - 1 )}  &\ll w_{(m-1,m)} \text{ and } a^{(k -1)}\ll D_{(m-1:4)} \\
			b^{(k)} a^{(k)} &\ll u_0
		\end{align*}
		so that $b^{(1)}   a^{(1)} b^{(2)}   a^{(2)}   \cdots   b^{(k)}   a^{(k)} $ is a subword of $w_2$.
		
		\item[-] Case 3:  $a \in \{a_{2s-1}, a_{2s}\}$, for $4 \le s \le k$. We have 
		\begin{align*}		
			b^{(1)}&\ll w_{(1,4)} \text{ and }  a^{(1)}  \ll D_{(1:6)}\\	
			\text{for all} \ i \in \{3,\ldots,s-2,s+1, \ldots, k\}, \quad b^{(i-1)} &\ll w_{(2i-1,2i)} \text{ and }\\ a^{(i-1)}  &\ll \begin{cases}
				D_{(2i-1:2i+2)},&\text{ if } i \ne k;\\
				D_{(m-1:4)},&\text{ if } i = k
			\end{cases}  \\
			b^{(s-2)} &\ll w_{(2s-3,2s-2)}  \\
			a^{(s-2)} b^{(s-1)} a^{(s-1)} &\ll w_{(2s-1,2s)}\\
			b^{( k)} a^{( k)} &\ll u_0
		\end{align*}
		so that  $ b^{(1)}  a^{(1)} b^{(2)}  a^{(2)} \cdots   b^{( k)}   a^{(k)} $ is a subword of $w_2$.
	\end{itemize}
	Hence, $a$ and $b$ alternate in the word $w_2$.
	
	Conversely, suppose $a$ and $b$ are not adjacent in $G$. We show that, for some $j \ge 1$, one of the following words is a subword of $w_2$:  $a^{(j)}b^{(j)} b^{(j+1)}a^{(j+1)}$,  $b^{(j)}a^{(j)} a^{(j+1)} b^{(j+1)}$,  $a^{(j)}a^{(j+1)} b^{(j)} b^{(j+1)}$, $b^{(j)}b^{(j+1)} a^{(j)} a^{(j+1)}$. We deal this in the following three cases:
	\begin{itemize}
		\item[-] Case 1: $a,b \in A$. We have the following subcases:
		
		\begin{itemize}
			\item Subcase 1.1:  $a, b$ belong to the same pair. 
			\begin{enumerate}[label=\roman*)]
				\item Let $a, b$ belong to the first pair. If $a, b \in \{a_1,a_2,a_3\}$, without loss of generality, assume $a = a_s$ and $b = a_t$ with $s > t$. Then  clearly $a^{(1)}b^{(1)}b^{(2)}a^{(2)} \ll w_{(1,4)}$. Otherwise, let $a \in \{a_1,a_2,a_3\}$ and $b = a_4$. Then  \[a^{(1)} b^{(1)} a^{(2)} b^{(2)} \ll w_{(1,4)}, \text{ and } 
				b^{(3)} a^{(3)} \ll 
				 \begin{cases}
				 	u_0, & \text{for } m = 6;\\
				 	D_{(5:8)}, & \text{for } m > 6,
				 \end{cases}	
				 \] so that $a^{(2)} b^{(2)} b^{(3)} a^{(3)} \ll w_2$. Thus, in any case,  $a^{(j)}b^{(j)} b^{(j+1)}a^{(j+1)} \ll w_2$, for some $j$.
				 \item If $a = a_5 $ and $b = a_6$ then $a^{(1)} b^{(1)} a^{(2)} b^{(2)} \ll w_{(5,6)}$ and
				 	\[
				 b^{(3)} a^{(3)} \ll 
				 \begin{cases}
				 	u_0, & \text{for } m = 6;\\
				 	D_{(7:4)}, & \text{for } m = 8; \\
				 	D_{(7:10)}, & \text{for } m > 8,
				 \end{cases}	
				 \]
				 so that $a^{(2)} b^{(2)} b^{(3)} a^{(3)} \ll w_2$.
				 \item For $4 \le s \le k$, if $a = a_{2s-1}$ and $b = a_{2s}$ then $a^{(s-2)} b^{(s-2)} a^{(s-1)} b^{(s-1)} \ll w_{(2s-1,2s)}$ and 
				 	\[
				 b^{(s-3)} a^{(s-3)} \ll 
				 \begin{cases}
				 	D_{(1:6)}, & \text{for } s = 4; \\
				 	D_{(2s-5:2s-2)}, & \text{for } s>4,
				 \end{cases}	
				 \]
				  so that $b^{(s-3)} a^{(s-3)} a^{(s-2)} b^{(s-2)} \ll w_2$.
			\end{enumerate}	
			\item Subcase 1.2: $a$ and $b$ belong to different pairs. 
			\begin{enumerate}[label=\roman*)]
				\item Let one of $a, b$ belong to the first pair $(a_1 \prec a_2 \prec a_3, a_4)$. Without loss of generality, let $a \in \{a_{1}, a_{2}, a_{3}, a_{4}\}$. As $b \notin \{a_{1}, a_{2}, a_{3}, a_{4}\}$, $b$ does not appear in $w_{(1,4)}$. Accordingly, we have $$a^{(1)}a^{(2)} \ll w_{(1,4)} \text{ so that }\; a^{(1)}a^{(2)}b^{(1)}b^{(2)} \ll w_2.$$
				\item If $a \in \{a_{2s-1}, a_{2s}\}$, for $3 \le s \le k $ then consider the following cases depending on whether $b$ is in the subsequent/preceding  pair or not:
				\begin{enumerate}
					\item If $b \in \{a_{2s+1}, a_{2s+2}\}$ then 
					\begin{align*}
						b^{(s-2)} &\ll \begin{cases}
							D_{(1:6)},&\text{ for } s = 3;\\
							D_{(2s-3:2s)},&\text{ for } s > 3
						\end{cases} \\
						a^{(s-2)} a^{(s-1)} &\ll w_{(2s-1, 2s)}\\
						b^{(s-1)} b^{(s)} &\ll w_{(2s+1, 2s+2)} 
					\end{align*}
					so that $b^{(s-2)}a^{(s-2)} a^{(s-1)}b^{(s-1)} \ll w_2$.
					
					\item If $b \notin \{a_{2s-3}, a_{2s-2},a_{2s+1}, a_{2s+2}\}$ then  	\[
					b^{(s-2)} a^{(s-2)} a^{(s-1)} b^{(s-1)}\ll
					\begin{cases}
						D_{(1:6)} w_{(5,6)}D_{(5:4)}, & \text{for } s = 3 \text{ and } m=6;\\
						D_{(1:6)} w_{(5,6)}D_{(5:8)}, & \text{for } s = 3 \text{ and } m>6;\\
						D_{(2s-3:2s)}w_{(2s-1,2s)}D_{(2s-1:4)}, & \text{for } s = k;\\
						D_{(2s-3:2s)}w_{(2s-1,2s)}D_{(2s-1:2s+2)}, & \text{for } 3 < s < k.
					
					\end{cases}	
					\]
					Thus, $	b^{(s-2)} a^{(s-2)} a^{(s-1)} b^{(s-1)} \ll w_2$.
				\end{enumerate}
			\end{enumerate}
	\end{itemize}
	
	\item[-] Case 2: $a \in A$ and $b \in B$. We deal this case in the following three subcases.
		\begin{itemize}
			\item Subcase 2.1: $a \in \{a_1,a_2,a_3,a_4\}$. 
			\begin{enumerate}[label=\roman*)]
				\item 	If $b \in N(a_4)$  then  $a^{(1)} a^{(2)} b^{(1)} \ll w_{(1,4)} $. This implies that  $a^{(1)} a^{(2)} b^{(1)} b^{(2)} \ll w_2 $.
				\item 	If $b \in N(a_3) \setminus N(a_4)$ then $ b^{(1)} a^{(1)} a^{(2)} \ll w_{(1,4)}$ and $b^{(2)} \ll w_{(5,6)}$. Thus, $b^{(1)} a^{(1)} a^{(2)} b^{(2)} \ll w_2$.
				\item If $b \not \in N(a_3) \cup N(a_4)$ then either $b^{(1)} \ll 	B_{*1}$ or $b^{(1)}\ll	B_{*4}'$, and   $b^{(2)} \ll w_{(5,6)}.$ Thus, either $b^{(1)} a^{(1)} a^{(2)} b^{(2)} \ll w_2$	 or $ a^{(1)} a^{(2)} b^{(1)}b^{(2)} \ll w_2$.
			\end{enumerate}
			\item Subcase 2.2: $a \in \{a_{2s-1},a_{2s}\}$, for $3 \le s \le k$.
			\begin{enumerate}[label=\roman*)]
				\item If $b \in N(a_{2s})$ then $b^{(s-1)} \ll B_{2s}' $. Thus, $$b^{(s-2)} \ll \begin{cases}
					w_{(1,4)}, &\text{ for } s=3;\\
					w_{(2s-3,2s-2)}, &\text{ for } s>3,
				\end{cases}  \text{ and } a^{(s-2)} a^{(s-1)} b^{(s-1)}  \ll w_{(2s-1,2s)},$$ so that $b^{(s-2)} a^{(s-2)} a^{(s-1)} b^{(s-1)} \ll w_2$.
				\item If $b \in N(a_{2s-1})$ then $b^{(s-1)} \ll B_{2s-1}$. Thus, $$b^{(s-2)} \ll \begin{cases}
					w_{(1,4)}, &\text{ for } s=3;\\
					w_{(2s-3,2s-2)}, &\text{ for } s>3,
				\end{cases}  \text{ and } b^{(s-1)}  a^{(s-2)} a^{(s-1)}  \ll w_{(2s-1,2s)},$$ so that $b^{(s-2)} b^{(s-1)} a^{(s-2)} a^{(s-1)}  \ll w_2$.
				\item If $b \not \in N(a_{2s-1}) \cup N(a_{2s})$ then either $b^{(s-1)} \ll B_{*(2s-1)}$ or $b^{(s-1)} \ll B_{*(2s)}'$. Thus, $$b^{(s-2)} \ll \begin{cases}
					w_{(1,4)}, &\text{ for } s=3;\\
					w_{(2s-3,2s-2)}, &\text{ for } s >3,
				\end{cases}$$ so that either $b^{(s-2)}b^{(s-1)}a^{(s-2)} a^{(s-1)}  \ll  w_2$ or $b^{(s-2)}a^{(s-2)} a^{(s-1)} b^{(s-1)} \ll w_2$.			
			\end{enumerate}					
		\end{itemize}
		
		\item[-] Case 3: $a,b \in B$.	The word $w_2$ is of the form  $w_{(1,4)}vu_0$. Clearly, $w_{(1,4)}|_{B} \ll w_{(1,4)}$, and by construction of $u_0$ we have, ${w_{(1,4)}^R}|_B \ll u_0$. Since $a$ and $b$ occur exactly once in $w_{(1,4)}$, without loss of generality, assume that $ab\ll w_{(1,4)}$ so that $w_2|_{{\{a,b\}}} = ab \cdots ba$. Since $w_2$ is a uniform word, both $a$ and $b$ must occur the same number of times. Thus, two $b$'s occur consecutively in $w_2|_{{\{a,b\}}}$. Thus, $a$ and $b$ do not alternate in $w_2$.
	\end{itemize}
	Hence, $w_2$ represents $G$ so that $\mathcal{R}(G) \le k$. \qed
\end{proof}

In view of Theorem \ref{length_three}, we note that the conjecture by Glen et al. holds good for all bipartite graphs considered in this section having $h(P_A) \ge 3$. For another special class from the remaining graphs, we have the following theorem.      

\begin{theorem}\label{length_two}
	Let $G$ be a bipartite graph with a partite set $A = \{a_1,a_2,\ldots,a_m\}$, where $m =2k$ ($k\ge3$). If $h(P_A) = 2$ and contains at least two disjoint chains of size two then $\mathcal{R}(G) \le k$.
\end{theorem}
\begin{proof}
	Let $a_1 \prec a_2$ and $a_4 \prec a_5$ be two disjoint chains in $P_A$. We extend the method given in Theorem \ref{length_three}, and construct a $k$-uniform word-representant of $G$. First partition the set $A$ into $k-1$ pairs given by:
	\[(a_1 \prec a_2, a_3), (a_4 \prec a_5, a_6), (a_7, a_8), \ldots, (a_{m-1}, a_m)\]

Except for the first two pairs, as given in Section \ref{const_single}, we construct the word $w_{(i,j)}$ for each pair $(a_i, a_j)$, for $7 \le i \le m-1$ and $i$ odd. For the first pair $(a_1 \prec a_2, a_3)$, construct the word denoted by $w_{(1,3)}$ on the set $\{a_1,a_2,a_3\} \cup B$ as per the following:
		\[w_{(1,3)} = B_{*1}  a_2  B_2  a_1  B_1  a_3  C_{1}  a_1  C_2  a_2   B_3'  a_3  B_{*3}' \]
		where,
		\begin{enumerate}[label=\roman*.]
			\item $B_{*1}B'_{*3}$: A permutation on the vertices of $B$ which are not adjacent to any vertex of $\{a_1,a_2,a_3\}$.
			\item $B_{1}$: A permutation on the neighbors of $a_1$ but not of  $a_3$, i.e.,  $N(a_{1}) \setminus  N(a_3)$.
			\item $B_{2}$: A permutation on the neighbors of $a_2$ but not of  $a_1$ and $a_3$, i.e.,  $N(a_{2}) \setminus ( N(a_1) \cup N(a_3))$.		
			\item $C_1$: A permutation on the common neighbors of  $a_1$ and $a_3$, i.e., $N(a_{1}) \cap N(a_{3})$. 
			\item $C_2$: A permutation on the common neighbors of  $a_3$ and $a_2$ but not of $a_{1}$, i.e., $(N(a_{2}) \cap N(a_{3})) \setminus N(a_{1})$. 
			\item $B_{3}'$: A permutation on the neighbors of $a_3$ but not of $a_2$, i.e.,  $N(a_{3}) \setminus N(a_{2})$. 
		\end{enumerate}
	Similarly, for the pair $(a_4 \prec a_5, a_6)$, construct the word denoted by $w_{(4,6)}$ on the set $\{a_4,a_5,a_6\} \cup B$ as per the following:
		\[w_{(4,6)} = B_{*4}   a_5  B_5  a_4  B_4  a_6  C_{4}  a_4  C_5  a_5   B_6'  a_6  B_{*6}' \]
		where,
		\begin{enumerate}[label=\roman*.]
			\item $B_{*4}B'_{*6}$: A permutation on the vertices of $B$ which are not adjacent to any vertex of $\{a_4,a_5,a_6\}$.
			\item $B_{4}$: A permutation on the neighbors of $a_4$ but not of  $a_6$, i.e.,  $N(a_{4}) \setminus  N(a_6)$.
			\item $B_{5}$: A permutation on the neighbors of $a_5$ but not of  $a_4$ and $a_6$, i.e.,  $N(a_{5}) \setminus (N(a_4) \cup N(a_6))$.		
			\item $C_1$: A permutation on the common neighbors of  $a_4$ and $a_6$, i.e., $N(a_{4}) \cap N(a_{6})$. 
			\item $C_2$: A permutation on the common neighbors of  $a_6$ and $a_5$ but not of $a_{4}$, i.e., $(N(a_{5}) \cap N(a_{6})) \setminus N(a_{4})$. 
			\item $B_{6}'$: A permutation on the neighbors of $a_6$ but not of $a_5$, i.e.,  $N(a_{6}) \setminus N(a_{5})$. 
		\end{enumerate}

Now construct the permutation $v_0$ on the set $A \cup B$ by considering the reversal of the word $w_{(1,3)}$ restricted to the vertices of $B$ as per the following:
		$$v_0 = a_3a_1a_2 w_{(1,3)}^R|_{B}a_6a_4 a_5a_8a_7\cdots {a_m}a_{m-1}$$
Finally, construct the $k$-uniform word $w_3$ over $A \cup B$ as per the following:
		\[w_3 = w_{(1,3)}D_{(1:6)}  w_{(4,6)} D_{(4:8)} w_{(7,8)} \cdots w_{(m-1,m)} D_{(m-1:3)}v_0\]

Along the similar lines of  Theorem \ref{length_three}, we show that the word $w_3$ represents the graph $G$. 
Suppose $a$ and $b$ are adjacent in $G$. Let $a \in A$ and $b \in B$. We deal this in the following four cases:
	\begin{itemize}
		\item[-] Case 1: $a \in \{a_1,a_2,a_3\}$: We have
		\begin{align*}
			a^{(1)}b^{(1)}a^{(2)} &\ll w_{(1,3)} \\
			b^{(2)} &\ll w_{(4,6)} \text{ and } a^{(3)}\ll D_{(4:8)}\\
		\text{for all } \ 4 \le i \le k-1, \quad	b^{(i-1)} &\ll w_{(2i-1,2i)} \text{ and } a^{(i)} \ll D_{(2i-1:2i+2)} \\
			b^{(k -1)} &\ll w_{{(m-1,m)}}\\
			a^{( k)} b^{( k)} &\ll v_0
		\end{align*}
		so that $ a^{(1)}   b^{(1)}   a^{(2)}   b^{(2)}   \cdots   a^{(k)}  b^{(k)} \ll w_3$.
		
		\item[-] Case 2: $a \in \{a_4,a_5,a_6\}$: We have
		\begin{align*}
		b^{(1)} &\ll w_{(1,3)}\\
		a^{(1)} b^{(2)} a^{(2)} &\ll w_{(4,6)}  \\		
	 \text{for all} \ 4 \le i \le k-1, \quad	b^{(i-1)}&\ll w_{(2i-1,2i)} \text{ and } a^{(i-1)}  \ll D_{(2i-1:2i+2)}  \\
		b^{(k -1)}&\ll w_{{(m-1,m)}} \text{ and }   a^{(k-1)} \ll D_{(m-1:3)}\\
		b^{( k)} a^{( k)} &\ll v_0
		\end{align*}
		so that $   b^{(1)} a^{(1)}   b^{(2)}   a^{(2)}   \cdots   b^{(k)}   a^{(k)} \ll w_3$.
		
		\item[-] Case 3: $a \in \{a_{7}, a_{8}\}$: We have
		\begin{align*}
			b^{(1)}  &\ll w_{(1,3)} \text{ and } a^{(1)}\ll D_{(1:6)}\\
			b^{(2)} &\ll w_{(4,6)}\\
			a^{(2)} b^{(3)} a^{(3)} &\ll w_{(7,8)} \\			
		\text{ for all } 5 \le i \le k-1, \quad	b^{(i-1)}	&\ll w_{(2i-1,2i)}  \text{ and } 	a^{(i-1)}\ll D_{(2i-1:2i+2)} \\
			b^{(k - 1 )} &\ll w_{(m-1,m)} \text{ and } a^{(k -1)} \ll D_{(m-1:3)}\\
			b^{(k)} a^{(k)} &\ll v_0
		\end{align*}
		so that $  b^{(1)}   a^{(1)}   b^{(2)}   a^{(2)}   \cdots   b^{(k)}   a^{(k)}  \ll w_3$.
		
		\item[-] Case 4: $a \in \{a_{2s-1}, a_{2s}\}$, for $5 \le s \le k$. We have 
		\begin{align*}		
			b^{(1)}  &\ll w_{(1,3)} \text{ and } a^{(1)} \ll D_{(1:6)}\\	
			b^{(2)} &\ll w_{(4,6)} \text{ and }  a^{(2)}\ll D_{(4:8)}\\	
			\text{for all } \ i \in \{4,\ldots,s-2,s+1, \ldots, k\}, \quad b^{(i-1)} &\ll w_{(2i-1,2i)} \text{ and } \\ a^{(i-1)}&\ll \begin{cases}
				D_{(2i-1:2i+2)},&\text{ if } i \ne k;\\
				D_{(m-1:3)},&\text{ if } i = k
			\end{cases}  \\
			b^{(s-2)} &\ll w_{(2s-3,2s-2)} \\
			a^{(s-2)} b^{(s-1)} a^{(s-1)} &\ll w_{(2s-1,2s)}\\
			b^{( k)} a^{( k)} &\ll v_0
		\end{align*}
		so that $ b^{(1)}   a^{(1)} b^{(2)}   a^{(2)}  \cdots  b^{( k)}  a^{(k)} \ll w_3 $.
	\end{itemize}
	Hence, $a$ and $b$ alternate in the word $w_3$. 	
	Conversely, suppose $a$ and $b$ are not adjacent in $G$. 
	\begin{itemize}
		\item[-] Case 1: $a,b \in A$. We have the following subcases:
		\begin{itemize}
			\item Subcase 1.1: $a$ and $b$ belong to the same pair.
			\begin{enumerate}[label=\roman*)]
				\item Let $a, b$ belong to the first pair. If $a, b \in \{a_1,a_2\}$, without loss of generality, assume $a = a_2$ and $b = a_1$. Then  clearly $a^{(1)}b^{(1)}b^{(2)}a^{(2)} \ll w_{(1,3)}$. Otherwise, let $a \in \{a_1,a_2\}$ and $b = a_3$. Then $a^{(2)}b^{(2)}  \ll w_{(1,3)}$ and 
				\[
				b^{(3)} a^{(3)} \ll 
				\begin{cases}
					v_0, & \text{for } m = 6; \\
					D_{(4:8)}, & \text{for } m > 6, 
				\end{cases}	
				\]
				so that $a^{(2)} b^{(2)} b^{(3)} a^{(3)} \ll w_3$. Thus, in any case,  $a^{(j)}b^{(j)} b^{(j+1)}a^{(j+1)} \ll w_3$, for some $j$.
				
				\item Let $a, b$ belong to the second pair. If $a, b \in \{a_4,a_5\}$, without loss of generality, assume $a = a_5$ and $b = a_4$. Then  clearly $a^{(1)}b^{(1)}b^{(2)}a^{(2)} \ll w_{(4,6)}$. Otherwise,  let $a \in \{a_4, a_5\}$ and $b = a_6$. Then $ a^{(2)} b^{(2)}\ll w_{(4,6)}$ and 
				\[
				b^{(3)} a^{(3)} \ll 
				\begin{cases}
					v_0, & \text{for } m = 6; \\
					D_{(7:3)}, & \text{for } m = 8; \\
					D_{(7:10)} & \text{for } m > 8,
				\end{cases}	
				\]
				so that  $a^{(2)} b^{(2)} b^{(3)} a^{(3)} \ll w_3$. Thus, in any case,  $a^{(j)}b^{(j)} b^{(j+1)}a^{(j+1)} \ll w_3$, for some $j$.
				
				\item Let $a, b \in \{a_7, a_8\}$. Without loss of generality, assume $a = a_7$ and $b = a_8$. Then $b^{(1)} a^{(1)} \ll D_{(1:6)}$ and $a^{(2)} b^{(2)}\ll w_{(7,8)}$ so that   $b^{(1)} a^{(1)}a^{(2)} b^{(2)} \ll w_3$.
				
				\item Let $a, b \in \{a_9, a_{10}\}$. Assume $a = a_9$ and $b = a_{10}$. Then  $b^{(2)} a^{(2)} \ll D_{(4:8)}$ and $a^{(3)} b^{(3)} \ll w_{(9,10)}$ so that $b^{(2)} a^{(2)}a^{(3)} b^{(3)} \ll w_3$.
					
				\item For $6 \le s \le k$, if $a = a_{2s-1}$ and $b = a_{2s}$ then  $
				b^{(s-3)} a^{(s-3)} \ll 
				D_{(2s-5:2s-2)}$ and $a^{(s-2)} b^{(s-2)} \ll w_{(2s-1,2s)}$.
				Thus, $b^{(s-3)} a^{(s-3)} a^{(s-2)} b^{(s-2)} \ll  w_3$.
			\end{enumerate}
			\item Subcase 1.2: $a$ and $b$ belong to different pairs.
			\begin{enumerate}[label=\roman*)]
				\item Let one of $a,b$ be in the first pair  $(a_1 \prec a_2,a_3)$. Without loss of generality, assume $a \in \{a_{1}, a_{2}, a_3\}$. Then $a^{(1)}a^{(2)}\ll w_{(1,3)}$ and $b^{(1)} \not \ll w_{(1,3)}$. Thus, $a^{(1)}a^{(2)}b^{(1)}b^{(2)} \ll w_3.$
				
				\item Let one of $a,b$ be in the second pair $(a_4 \prec a_5,a_6)$ and assume $a \in \{a_{4}, a_{5}, a_6\}$. Then $b^{(1)} \ll D_{(1:6)}$ and $ a^{(1)} a^{(2)} \ll w_{(4,6)}$. Thus, $b^{(1)} a^{(1)} a^{(2)} b^{(2)} \ll w_3$.
				\item For $ 4 \le s\le k$, if $a \in \{a_{2s-1}, a_{2s}\} $ then consider the following cases depending on whether $b$ is in the subsequent/preceding pair or not:
				\begin{enumerate}
					\item If $b \in \{a_{2s+1},a_{2s+2}\}$ then  
					\begin{align*}
						b^{(s-2)} &\ll \begin{cases}
							D_{(4:8)}, & \text{ for } s = 4;\\
							D_{(2s-3:2s)}, & \text{ for } s>4
						\end{cases} \\
						a^{(s-2)} a^{(s-1)} &\ll w_{(2s-1,2s)}\\
						b^{(s-1)} b^{(s)} &\ll w_{(2s+1,2s+2)}
					\end{align*}
					so that $b^{(s-2)}a^{(s-2)}a^{(s-1)} b^{(s-1)} \ll w_3$.
					
					\item  If $b \notin \{a_{2s-3}, a_{2s-2},a_{2s+1}, a_{2s+2}\}$ then 
					\begin{align*}
						b^{(s-2)} &\ll \begin{cases}
							D_{(4:8)},&\text{ for } s = 4;\\
							D_{(2s-3:2s)},&\text{ for } s > 4
						\end{cases}\\
						 a^{(s-2)} a^{(s-1)} &\ll w_{(2s-1,2s)} \\
						 b^{(s-1)} &\ll \begin{cases}
						 	D_{(2s-1:2s+2)}, &\text{ for } s \ne k;\\
						 	D_{(m-1:3)}, &\text{ for } s = k
						 \end{cases}
					\end{align*}					
					so that $b^{(s-2)} a^{(s-2)} a^{(s-1)} b^{(s-1)} \ll w_3$.
				\end{enumerate}			
			\end{enumerate}		
		\end{itemize} 
		
		\item[-] Case 2: $a \in A$ and $b \in B$. We deal this case in the following three subcases.
		\begin{itemize}
			\item Subcase 2.1: $a \in \{a_1,a_2,a_3\}$. \begin{enumerate}[label=\roman*)]
				\item If $b \in N(a_3)$ then  $a^{(1)} a^{(2)} b^{(1)} \ll w_{(1,3)} $ so that  $a^{(1)} a^{(2)} b^{(1)} b^{(2)} \ll w_3 $. 
				\item If $b \in N(a_2) \setminus N(a_3)$ then $ b^{(1)} a^{(1)} a^{(2)}  \ll w_{(1,3)}$ and $b^{(2)} \ll  w_{(4,6)} $. Thus, $ b^{(1)} a^{(1)} a^{(2)} b^{(2)} \ll w_3$.
				\item  If $b \not \in N(a_2) \cup N(a_3)$ then either $b^{(1)} \ll B_{*1}$ or $b^{(1)} \ll B_{*3}'$, and $a^{(1)} a^{(2)} \ll w_{(1,3)}$. Thus, either $b^{(1)} a^{(1)} a^{(2)}  b^{(2)} \ll w_3$	 or $ a^{(1)} a^{(2)} b^{(1)}b^{(2)} \ll w_3$.
			\end{enumerate}
			
			\item Subcase 2.2: $a \in \{a_4,a_5,a_6\}$. \begin{enumerate}[label=\roman*)]
				\item If $b \in N(a_6)$ then $b^{(1)} \ll w_{(1,3)}$  and $a^{(1)} a^{(2)} b^{(2)} \ll w_{(4,6)}$. Thus,  $b^{(1)} a^{(1)} a^{(2)} b^{(2)} \ll w_3 $.
				\item  If $b \in N(a_5) \setminus N(a_6)$ then $b^{(1)} \ll w_{(1,3)}$  and $b^{(2)} a^{(1)} a^{(2)}  \ll  w_{(4,6)}$. Thus,  $b^{(1)} b^{(2)} a^{(1)} a^{(2)}  \ll w_3 $.
				\item 	If $b \not \in N(a_5) \cup N(a_6)$ then $b^{(1)} \ll w_{(1,3)}$,  and either $b^{(2)} \ll B_{*4}$ or $b^{(2)} \ll	B_{*6}'$ so that either $b^{(1)}b^{(2)} a^{(1)} a^{(2)} \ll w_3$	or $ b^{(1)}a^{(1)} a^{(2)} b^{(2)} \ll w_3$.
			\end{enumerate}
			\item Subcase 2.3:  $a \in \{a_{2s-1},a_{2s}\}$, for $4 \le s \le k$. 
			\begin{enumerate}[label=\roman*)]
				\item If $b \in N(a_{2s})$ then  $b^{(s-2)} \ll \begin{cases}
					w_{(4,6)}, & \text{if } s = 4 ;\\
					w_{(2s-3)(2s-2)} ,& \text{if }  s >4,
				\end{cases}$ and $a^{(s-2)} a^{(s-1)} b^{(s-1)}\ll w_{(2s-1,2s)}$. Thus, $b^{(s-2)} a^{(s-2)} a^{(s-1)} b^{(s-1)}\ll w_3$.
				\item If $b \in N(a_{2s-1})$ then $b^{(s-2)} \ll \begin{cases}
					w_{(4,6)}, & \text{if } s = 4 ;\\
					w_{(2s-3)(2s-2)} ,& \text{if }  s >4, 
				\end{cases}$ and $b^{(s-1)} a^{(s-2)} a^{(s-1)} \ll w_{(2s-1,2s)}$. Thus, $b^{(s-2)}b^{(s-1)} a^{(s-2)} a^{(s-1)} \ll w_3$.
				\item If $b \not \in N(a_{2s-1}) \cup N(a_{2s})$ then 
				\begin{align*}
				 b^{(s-2)} &\ll \begin{cases}
					w_{(4,6)}, & \text{if } s = 4 ;\\
					w_{(2s-3)(2s-2)} ,& \text{if }  s >4,
				\end{cases}  			 				
				\end{align*}
				and either $b^{(s-1)} \ll B_{*(2s-1)}$ or $b^{(s-1)}	 \ll	B_{*(2s)}'$. Thus, either $b^{(s-2)}  b^{(s-1)} a^{(s-2)} a^{(s-1)}  \ll w_3$ or  $b^{(s-2)}a^{(s-2)} a^{(s-1)} b^{(s-1)}  \ll w_3$.
			\end{enumerate}	
		\end{itemize}
		
		\item[-] Case 3: $a,b \in B$.	The word $w_3$ is of the form  $w_{(1,3)}vv_0$. Clearly, $w_{(1,3)}|_{B} \ll w_{(1,3)}$, and by construction of $v_0$, we have ${w_{(1,3)}^R}|_B \ll v_0$. Since $a$ and $b$ occur exactly once in $w_{(1,3)}$, without loss of generality, assume $ab\ll w_{(1,3)}$ so that $w_3|_{{\{a,b\}}} = ab \cdots ba$. Since $w_3$ is a uniform word, both $a$ and $b$ must occur the same number of times. Thus, two $b$'s occur consecutively in $w_3|_{{\{a,b\}}}$. Thus, $a$ and $b$ do not alternate in $w_3$. 
	\end{itemize}
	Hence, $w_3$ represents $G$ so that $\mathcal{R}(G) \le k$.\qed
\end{proof}

\section{Conclusion}

In this paper, we provide an upper bound on the representation number of the bipartite graphs. We also improved the same for certain subclasses of the bipartite graphs. Using which, we established that the conjecture by Glen et al. on the representation number of bipartite graphs holds for most of the bipartite graphs. The only class of bipartite graphs for which the conjecture remains open is the bipartite graphs with partite sets of equal and even size. Even in this class, we proved the conjecture for some special subclasses using the neighborhood inclusion graph approach. In this approach, we found that the conjecture can be verified for some other special classes too. Accordingly, we believe that the conjecture holds for the remaining bipartite graphs. In the neighborhood graph approach, we have considered a specific chain cover of the poset associated with a partite set of a bipartite graph. By considering a smallest chain cover and adopting our method suitably, one may improve the upper bound on the representation number of various classes of bipartite graphs.

\appendix

\section{Demonstration of Construction given in Section \ref{const_single}}
\label{demo}

 	\begin{figure}[!h]
 		\centering
 		\begin{minipage}{.55\textwidth}
 			\centering
 			\begin{tikzpicture}[scale=0.5]				
 				\node (A) at (-0.2,0) [label=left:$B$] {};				
 				\node (B) at (-0.2,2) [label=left:$A$] {};			
 				\vertex (1) at (0,0) [label=below:$b_1$] {};  
 				\vertex (2) at (1,0) [label=below:$b_2$] {};
 				\vertex (3) at (2,0) [label=below:$b_3$] {};
 				\vertex (4) at (3,0) [label=below:$b_4$] {};
 				\vertex (5) at (4,0) [label=below:$b_5$] {};  
 				\vertex (6) at (5,0) [label=below:$b_6$] {};
 				\vertex (7) at (0,2) [label=above:$a_1$] {};
 				\vertex (8) at (1,2) [label=above:$a_2$] {};
 				\vertex (9) at (2,2) [label=above:$a_3$] {};
 				\vertex (10) at (3,2) [label=above:$a_4$] {};
 				\vertex (11) at (4,2) [label=above:$a_5$] {};  				
 				\path				
 				(1) edge (8)
 				(1) edge (9)
 				(2) edge (9)
 				(2) edge (10)
 				(3) edge (11)
 				(3) edge(7)				
 				(4) edge (7)
 				(4) edge (8)
 				(4) edge (9)
 				(4) edge (10)
 				(5) edge (8)
 				(5) edge (9)
 				(5) edge (10)
 				(5) edge (11)
 				(6) edge (7)
 				(6) edge (8)
 				(6) edge (9)	
 				(6) edge (10)
 				(6) edge (11);		
 			\end{tikzpicture}
 			\caption{A bipartite graph}
 			\label{fig_G1}
 		\end{minipage}
 	\end{figure}
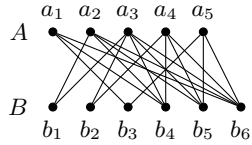
 	
 		Let $G = (A \cup B, E)$ be the bipartite graph given in Fig. \ref{fig_G1}. Consider the smaller partite set $A = \{a_1,a_2,a_3,a_4,a_5\}$. Since $A$ is of odd size,  add an isolated vertex  $a_6$ to the graph $G$ such that $N(a_6) = \varnothing$. Let $A_1 = A \cup \{a_6\}$.
 		 
 	\begin{enumerate}[label=\Roman*.]
 	\item Partition the set $A_1$ into pairs, i.e., $(a_1,a_2), (a_3,a_4), (a_5,a_6)$.
 	\item For each pair, we construct the following words:
 	\begin{enumerate}[label=\roman*.]
 		\item We have, $N(a_1) = \{b_3,b_4,b_6\}$ and $N(a_2) = \{b_1,b_4,b_5,b_6\}$. As per the construction, $B_{*1} = b_2$, $B_1=b_3$, $w_{(1,2)}'= b_4 b_6$, $B_{2}'=b_1b_5$ and $B_{*2}'$ is an empty word. Thus,  $$w_{(1,2)} = b_2   a_1  b_3   a_2   b_4   b_6  a_1   b_1   b_5  a_2.$$
 		\item We have, $N(a_3) = \{b_1,b_2,b_4,b_5,b_6\}$ and $N(a_4) = \{b_2,b_4,b_5,b_6\}$. As per the construction,  $B_{*3} = b_3$, $B_3=b_1$, $w_{(3,4)}'= b_2 b_4 b_5 b_6$, $B_{4}'$ and $B_{*4}'$ are empty words. Thus, $$w_{(3,4)} = b_3   a_3  b_1  a_4   b_2   b_4   b_5   b_6   a_3  a_4.$$
 		\item We have, $N(a_5) = \{b_3,b_5,b_6\}$ and $N(a_6) = \varnothing$. As per construction, $B_{*5} = b_1b_2b_4$, $B_5= b_3 b_5 b_6$, $w_{(5,6)}'$, $B_{6}'=b_1b_5$ and $B_{*6}'$ are empty words. Thus,
 		 $$w_{(5,6)} = b_1   b_2  b_4   a_5  b_3    b_5   b_6   a_6    a_5   a_6.$$
 	\end{enumerate}
 	Note that is in this process, if $b\in B$ is not in the neighborhood of a pair $( a_{2s-1},a_{2s})$, for some $s \in \{1,2,3\}$, we put $b$ in $B_{*(2s-1)}$ of  $w_{(2s-1,2s)}$.
 	\item Construct the permutations, $D_{(1:4)}, D_{(3:6)}$ and $D_{(5:2)}$, i.e., 
 	\begin{align*}
 		D_{(1:4)} & = a_6   a_5\\
 		D_{(3:6)} & = a_2 a_1\\
 		D_{(5:2)} & = a_4  a_3
 	\end{align*}
 	\item Construct the word $w_0$, i.e., \begin{align*}
 		w_0 &= a_2   a_1  w_{(1,2)}^R|_B   a_4    a_3   a_6   a_5\\
 		w_0 &= a_2 a_1  b_5  b_1 b_6 b_4 b_3  b_2 a_4 a_3  a_6 a_5
 	\end{align*}

 	\item Finally, we have the word $ w = w_{(1,2)}   D_{(1:4)}  w_{(3,4)}   D_{(3:6)}   w_{(5,6)}  D_{(5:2)}   w_0$  which represents $G \cup \{a_6\}$. 	
	$$ w  =  {\scriptsize b_2   a_1  b_3   a_2    b_4   b_6  a_1   b_1   b_5  a_2  a_6  a_5   b_3   a_3  b_1   a_4   b_2  b_4 b_5  b_6  a_3  a_4  a_2 a_1 b_1  b_2  b_4  a_5 b_3  b_5  b_6  a_6  a_5  a_6 a_4  a_3 a_2  a_1  b_5  b_1  b_6 b_4 b_3 b_2 a_4  a_3  a_6 a_5.}$$	
	
	\item Remove the vertex $a_6$ from $w$, we see that 
	
	$$w|_{(A\cup B)} = {\scriptsize b_2   a_1  b_3   a_2    b_4   b_6  a_1   b_1   b_5  a_2     a_5   b_3   a_3  b_1   a_4   b_2  b_4 b_5  b_6  a_3  a_4  a_2 a_1 b_1  b_2  b_4  a_5 b_3  b_5  b_6     a_5    a_4  a_3 a_2  a_1  b_5  b_1  b_6 b_4 b_3 b_2 a_4  a_3    a_5} $$
	is a word-representant of the bipartite graph $G$.
 	\end{enumerate}
\end{document}